\documentclass[a4paper,11pt,leqno]{smfart}

\usepackage{amssymb,amsmath,enumerate,verbatim,mathrsfs,graphics,graphicx,mathptm,float,mathtools,rotating}
\usepackage[T1]{fontenc}
\usepackage[latin1]{inputenc}
\usepackage[english, francais]{babel}
\usepackage[all]{xy}
\usepackage[pdfstartview=FitH,
            colorlinks=true,
            linkcolor=blue,
            urlcolor=blue,
            citecolor=red,
            bookmarks,
            bookmarksopen=true,
            bookmarksnumbered=true]{hyperref}

\usepackage{cleveref}
\theoremstyle{plain}

\newtheorem{theoalph}{Theorem}

\newtheorem{thmalph}[theoalph]{Theorem}
\newtheorem{propalph}[theoalph]{Proposition}
\newtheorem{coralph}[theoalph]{Corollary}

\theoremstyle{definition}

\theoremstyle{remark}

\theoremstyle{plain}
\newtheorem{thmsec}{Theorem}[section]
\newtheorem{thm}[thmsec]{Theorem}
\newtheorem{pro}[thmsec]{Proposition}
\newtheorem{lem}[thmsec]{Lemma}
\newtheorem{cor}[thmsec]{Corollary}

\newtheorem{conj}{Conjecture}

\theoremstyle{definition}

\theoremstyle{remark}
\newtheorem{rem}[thmsec]{Remark}

\def\og{\leavevmode\raise.3ex\hbox{$\scriptscriptstyle\langle\!\langle$~}}
\def\fg{\leavevmode\raise.3ex\hbox{~$\!\scriptscriptstyle\,\rangle\!\rangle$}}

\addtocounter{section}{0}             
\numberwithin{equation}{section}       

\newcommand{\C}{\mathbb{C}}

\newcommand{\pp}{\mathbb{P}^{2}_{\mathbb{C}}}

\newcommand\Sing{\mathrm{Sing}}

\DeclareMathAlphabet{\mathcalsansmathptm}{OMS}{cmsy}{m}{n} 
\newcommand\F{\mathcal{F}}

\newcommand\omegahesse{\omega_{\scalebox{0.64}{\ensuremath H}}^{4}}
\newcommand\omegahilbertcinq{\omega_{\scalebox{0.64}{\ensuremath H}}^{5}}
\newcommand\omegahilbertsept{\omega_{\scalebox{0.64}{\ensuremath H}}^{7}}
\newcommand\Gunderline{{\mspace{2mu}\underline{\mspace{-2mu}\mathcal{G}\mspace{-2mu}}\mspace{2mu}}}

\newcommand\omegaoverline{{\mspace{2mu}\overline{\mspace{-1.4mu}\omega\mspace{-1.4mu}}\mspace{2mu}}}

\newcommand\Hesse{\mathcal{F}_{\hspace{-0.4mm}\raisebox{-0.2mm}{\tiny{$H$}}}^{4}}
\newcommand\Hilbertcinq{\mathcal{F}_{\hspace{-0.4mm}\raisebox{-0.2mm}{\tiny{$H$}}}^{5}}
\newcommand\Hilbertsept{\mathcal{F}_{\hspace{-0.4mm}\raisebox{-0.2mm}{\tiny{$H$}}}^{7}}
\newcommand\elltext{\scalebox{1.1}{\ensuremath \ell}}
\newcommand\ellindice{\scalebox{0.9}{\ensuremath \ell}}

\begin{document}
\selectlanguage{english}

\title[Convex foliations of degree $5$ on the complex projective plane]{Convex foliations of degree $5$ on the complex projective plane}
\date{\today}

\author{Samir \textsc{Bedrouni}}

\address{Facult\'e de Math\'ematiques, USTHB, BP $32$, El-Alia, $16111$ Bab-Ezzouar, Alger, Alg\'erie}
\email{sbedrouni@usthb.dz}

\author{David \textsc{Mar\'{\i}n}}

\thanks{D. Mar\'{\i}n acknowledges financial support from the Spanish Ministry of Science, Innovation and Universities, through grants MTM2015-66165-P and PGC2018-095998-B-I00 and by the Agency for Management of University and Research Grants of Catalonia through the grant 2017SGR1725.}

\address{Departament de Matem\`{a}tiques, Universitat Aut\`{o}noma de Barcelona, E-08193  Cerdanyola del Vall\`es (Barcelona) Spain\\
Centre de Recerca Matem\`atica, Edifici Cc, Campus de Bellaterra, E-08193  Cerdanyola del Vall\`es (Barcelona) Spain}

\email{davidmp@mat.uab.es}


\maketitle{}

\begin{abstract}
We show that up to automorphisms of $\pp$ there are $14$ homogeneous convex foliations of degree $5$ on~$\pp.$ We establish some properties of the \textsc{Fermat} foliation $\F_{0}^{d}$ of degree $d\geq2$ and of the \textsc{Hilbert} modular foliation $\Hilbertcinq$ of degree~$5.$ As a consequence, we obtain that every
reduced convex foliation of degree $5$ on $\pp$ is linearly conjugated to one of the two foliations $\F_{0}^{5}$ or $\Hilbertcinq,$ which is a partial answer to a question posed in $2013$ by D.~\textsc{Mar\'{\i}n} and J.V.~\textsc{Pereira}. We end with two conjectures about the \textsc{Camacho-Sad} indices along the line at~infinity at the non radial singularities of the homogeneous convex foliations of degree $d\geq2$ on $\pp.$

\end{abstract}

\section{Introduction and statements of results}
\bigskip

\noindent This article is part of a series of works by the authors \cite{Bed17,BM18,BM19Z} on holomorphic foliations on the complex projective plane. For the definitions and notations used (radial singularities, \textsc{Camacho}-\textsc{Sad} index $\mathrm{CS}(\mathcal{F},\elltext,s),$ homogeneous foliations, etc.) we refer to~\cite[Sections~1~and~2]{BM18}.

\noindent Following \cite{MP13} a foliation on the complex projective plane is said to be \textsl{convex} if its leaves other than straight lines have no inflection points. Notice (see \cite{Per01}) that if $\F$ is a foliation of degree $d\geq 1$ on $\pp,$ then $\F$ cannot have more than $3d$ (distinct) invariant lines. Moreover, if this bound is reached, then $\F$ is necessarily convex; in this case $\F$ is said to be \textsl{reduced convex}.

\noindent To our knowledge the only reduced convex foliations known in the literature are those presented in~\cite[Table~1.1]{MP13}: the \textsc{Fermat} foliation $\F_{0}^{d}$ of degree~$d$, the \textsc{Hesse} pencil $\Hesse$ of degree~$4$, the \textsc{Hilbert} modular foliation $\Hilbertcinq$ of degree~$5$ and the \textsc{Hilbert} modular foliation $\Hilbertsept$ of degree~$7$ defined in affine chart respectively by the $1$-forms
\[
\omegaoverline_{0}^{d}=(x^{d}-x)\mathrm{d}y-(y^{d}-y)\mathrm{d}x,
\]
\[
\omegahesse=(2x^{3}-y^{3}-1)y\mathrm{d}x+(2y^{3}-x^{3}-1)x\mathrm{d}y,
\]
\[
\omegahilbertcinq=(y^2-1)(y^2-(\sqrt{5}-2)^2)(y+\sqrt{5}x)\mathrm{d}x-(x^2-1)(x^2-(\sqrt{5}-2)^2)(x+\sqrt{5}y)\mathrm{d}y,
\]
\[
\omegahilbertsept=(y^3-1)(y^3+7x^3+1)y\mathrm{d}x-(x^3-1)(x^3+7y^3+1)x\mathrm{d}y.
\]
D.~\textsc{Mar\'{\i}n} and J.V.~\textsc{Pereira} \cite[Problem~9.1]{MP13} asked the following question: are there other reduced convex foliations? The answer in degree $2$, resp. $3$, resp. $4$, to this question is negative, thanks to \cite[Proposition~7.4]{FP15}, resp. \cite[Corollary~6.9]{BM18}, resp. \cite[Theorem~B]{BM19Z}. In this paper we show that the answer in degree $5$ to \cite[Problem~9.1]{MP13} is also negative. To do this, we follow the same approach as that described in degree~$4$ in~\cite{BM19Z}. It~mainly consists in using Proposition~3.2 of \cite{BM19Z} which allows to associate to every pair $(\F,\elltext),$ where $\F$ is a reduced convex foliation of degree $d$ on $\pp$ and $\elltext$ an invariant line of $\F,$ a homogeneous convex foliation $\mathcal{H}_{\F}^{\ellindice}$ of degree $d$ on $\pp$ belonging to the \textsc{Zariski} closure of the $\mathrm{Aut}(\pp)$-orbit of $\F,$ and then to study for~$d=5$ the set of foliations $\mathcal{H}_{\F}^{\ellindice}$ where $\elltext$ runs through the invariant lines of $\F.$

\noindent A homogeneous foliation $\mathcal{H}$ of degree $d$ on $\pp$ is given, for a suitable choice of affine coordinates $(x,y),$ by~a homogeneous $1$-form $\omega=A(x,y)\mathrm{d}x+B(x,y)\mathrm{d}y,$ where $A,B$ are complex homogeneous polynomials of degree $d$ with $\gcd(A,B)=1.$ By \cite{BM18} to such a foliation is associated the rational map $\Gunderline_{\mathcal{H}}\hspace{1mm}\colon\mathbb{P}^{1}_{\mathbb{C}}\rightarrow \mathbb{P}^{1}_{\mathbb{C}}$ defined by $$\Gunderline_{\mathcal{H}}([x:y])=[-A(x,y):B(x,y)].$$
Notice (see~\cite{BM18}) that a homogeneous foliation  $\mathcal{H}$ on $\pp$ is convex if and only if its associated map~$\Gunderline_{\mathcal{H}}$ is critically fixed, {\it i.e.} every critical point of $\Gunderline_{\mathcal H}$ is a fixed point  of $\Gunderline_{\mathcal H}$. More precisely, a homogeneous foliation $\mathcal{H}$ of degree $d$ on $\pp$ is convex of type $\mathcal{T}_\mathcal{H}=\sum_{k=1}^{d-1}r_{k}\cdot\mathrm{R}_k$ ({\it i.e.} having $r_{1}$, resp.~$r_{2},\ldots,$ resp.~$r_{d-1}$ radial singularities of order $1$, resp. $2,\ldots,$ resp. $d-1$, the  $\mathrm{R}_k$'s being just symbols) if and only if the map $\Gunderline_{\mathcal{H}}$ possesses  $r_{1}$, resp.~$r_{2},\ldots,$ resp.~$r_{d-1}$ fixed critical points of multiplicity $1$, resp. $2\ldots,$ resp. $d-1,$ with $\sum_{k=1}^{d-1}kr_{k}=2d-2.$

\noindent Using results of \cite[pages~79--80]{CGNPP15} on critically fixed rational maps of degree $5$ from $\mathbb{P}^{1}_{\mathbb{C}}$ to itself and studying the convexity of a homogeneous foliation $\mathcal{H}$ of degree $5$ on $\pp$ according to the shape of its type~$\mathcal{T}_\mathcal{H},$ we~obtain the classification, up to automorphisms of $\pp$, of homogeneous convex foliations of degree~$5$~on~$\pp.$
\begin{thmalph}\label{thmalph:class-homogene-convexe-5}
{\sl Up to automorphisms of  $\pp$ there are $14$ homogeneous convex foliations  $\mathcal{H}_1,\ldots,\mathcal{H}_{14}$ of degree $5$ on the complex projective plane. They are respectively described in affine chart by the following $1$-forms

\begin{enumerate}[]
\item $\omega_1\hspace{1mm}=y^5\mathrm{d}x-x^5\mathrm{d}y$;
\smallskip
\item $\omega_2\hspace{1mm}=y^2(10\hspace{0.2mm}x^3+10x^2y+5xy^2+y^3)\mathrm{d}x-x^4(x+5y)\mathrm{d}y$;
\smallskip
\item $\omega_3\hspace{1mm}=y^3(10\hspace{0.2mm}x^2+5xy+y^2)\mathrm{d}x-x^3(x^2+5xy+10y^2)\mathrm{d}y$;
\smallskip
\item $\omega_4\hspace{1mm}=y^4(5x-3y)\mathrm{d}x+x^4(3x-5y)\mathrm{d}y$;
\smallskip
\item $\omega_5\hspace{1mm}=y^3(5x^2-3y^2)\mathrm{d}x-2x^5\mathrm{d}y$;
\smallskip
\item $\omega_6\hspace{1mm}=y^3(220\hspace{0.2mm}x^2-165xy+36y^2)\mathrm{d}x-121x^5\mathrm{d}y$;
\smallskip
\item $\omega_7\hspace{1mm}=y^4\Big(\big(5-\sqrt{5}\big)x-2y\Big)\mathrm{d}x
                                                       +x^4\Big(\big(7-3\sqrt{5}\big)x-2\big(5-2\sqrt{5}\big)y\Big)\mathrm{d}y$;
\smallskip
\item $\omega_8\hspace{1mm}=y^4\Big(5\big(3-\sqrt{21}\big)x+6y\Big)\mathrm{d}x
                                                       +x^4\Big(3\big(23-5\sqrt{21}\big)x-10\big(9-2\sqrt{21}\big)y\Big)\mathrm{d}y$;
\smallskip
\item $\omega_9\hspace{1mm}=y^3\Big(2\big(5+a\big)x^2-\big(15+a\big)xy+6y^2\Big)\mathrm{d}x
                                                       -x^4\Big(\big(1-a\big)x+2ay\Big)\mathrm{d}y,$
                                                       where $a=\sqrt{\scalebox{0.83}{$5\big(4\sqrt{61}-31\big)$}}$;
\smallskip
\item $\omega_{10}=y^3\Big(2\big(5+\mathrm{i}b\big)x^2-\big(15+\mathrm{i}b\big)xy+6y^2\Big)\mathrm{d}x
                                                -x^4\Big(\big(1-\mathrm{i}b\big)x+2\mathrm{i}by\Big)\mathrm{d}y,$
                                                where $b=\sqrt{\scalebox{0.83}{$5\big(4\sqrt{61}+31\big)$}}$;
\smallskip
\item $\omega_{11}=y^3(5x^2-y^2)\mathrm{d}x+x^3(x^2-5y^2)\mathrm{d}y$;
\smallskip
\item $\omega_{12}=y^3(20\hspace{0.2mm}x^2-5xy-y^2)\mathrm{d}x+x^3(x^2+5xy-20y^2)\mathrm{d}y$;
\smallskip
\item $\omega_{13}=y^2(5x^3-10\hspace{0.2mm}x^2y+10\hspace{0.2mm}xy^2-4y^3)\mathrm{d}x-x^5\mathrm{d}y$;
\smallskip
\item $\omega_{14}=y^3\Big(u(\sigma)x^2+v(\sigma)xy+w(\sigma)y^2\Big)\mathrm{d}x+\sigma\hspace{0.3mm}x^4\Big(2\sigma\big(\sigma^2-\sigma+1\big)x
                                                -\big(\sigma+1\big)\big(3\sigma^2-5\sigma+3\big)y\Big)\mathrm{d}y,$
\smallskip
\item []                                        where $u(\sigma)=(\sigma^2-3\sigma+1)(\sigma^2+5\sigma+1),$\hspace{1.5mm}
                                                      $v(\sigma)=-2(\sigma+1)(\sigma^2-5\sigma+1),$\hspace{1.5mm}
                                                      $w(\sigma)=(\sigma^2-7\sigma+1),$
\item []                                              $\sigma=\rho+\mathrm{i}\sqrt{\frac{1}{6}-\frac{4}{3}\rho-\frac{1}{3}\rho^2}$\,
                                                and   $\rho$ is the unique real number satisfying $8\rho^3-52\rho^2+134\rho-15=0.$
\end{enumerate}
}
\end{thmalph}
\vspace{2mm}

\noindent During the proof of this theorem in \S\ref{sec:Demonstration des principaux resultats} we also obtain the following dual result.

\begin{thmalph}
{\sl Up to conjugation by a \textsc{Möbius} transformation there are $14$ critically fixed rational maps of degree $5$ from the \textsc{Riemann} sphere to itself, namely the maps $\Gunderline_{\mathcal{H}_{\hspace{0.1mm}1}},\ldots,\Gunderline_{\mathcal{H}_{\hspace{0.1mm}14}}.$
}
\end{thmalph}
\smallskip

\noindent To every foliation $\F$ on $\pp$ and to every integer $d\geq2,$ we associate respectively the following two subsets of $\C\setminus\{0,1\}$:
\begin{itemize}
\item $\mathcal{CS}(\F)$ is, by definition, the set of $\lambda\in\C\setminus\{0,1\}$ for which there is a line $\elltext$ invariant by $\F$ and a non-degenerate singular point $s\in\elltext$ of $\F$ such that $\mathrm{CS}(\F,\elltext,s)=\lambda$;

\item $\mathcal{HCS}_d$ is defined as the set of  $\lambda\in\C\setminus\{0,1\}$ for which there exist two homogeneous convex foliations $\mathcal{H}$ and $\mathcal{H}'$ of degree $d$ on $\pp$ having respective singular points $s$ and $s'$ on the line at infinity $\elltext_{\infty}$ such that $\mathrm{CS}(\mathcal H,\elltext_{\infty},s)=\lambda$ and $\mathrm{CS}(\mathcal H',\elltext_{\infty},s')=\frac{1}{\lambda}.$
\end{itemize}
\smallskip

\noindent The following proposition, which will be proved in~\S\ref{sec:Demonstration des principaux resultats}, motivates the introduction of the sets  $\mathcal{CS}(\F)$ and~$\mathcal{HCS}_d.$
\begin{propalph}\label{propalph:CS(F)-CS(d)}
{\sl Let $\F$ be a reduced convex foliation of degree $d\geq2$ on $\pp.$ Then
\begin{itemize}
\item [\textit{(a)}] $\emptyset\neq\mathcal{CS}(\F)\subset\mathcal{HCS}_d;$
\item [\textit{(b)}] $\forall\hspace{1mm}\lambda\in\mathcal{CS}(\F)$,\hspace{1mm}$\frac{1}{\lambda}\in\mathcal{CS}(\F).$
\end{itemize}
}
\end{propalph}

\begin{rem}\label{rem:CS(Hilbertcinq)-CS(Fermat)} In particular, for the foliations $\Hilbertcinq$ and $\F_0^d,$ we have
\begin{itemize}
\item [$\bullet$] $\{-\frac{3}{2}\pm\frac{\sqrt{5}}{2}\}=\mathcal{CS}(\Hilbertcinq)\subset\mathcal{HCS}_5,$ \emph{cf.} \cite[Theorem~2]{MendesP05};
\item [$\bullet$] $\{(1-d)^{\pm 1}\}=\mathcal{CS}(\F_0^d)\subset\mathcal{HCS}_d$ for any $d\geq2,$ \emph{cf.} \cite[Example~6.5]{BM18}.
\end{itemize}
\end{rem}

\noindent The following theorem gives equivalent conditions for a foliation of degree $d\geq2$ on $\pp$ to be conjugated to the \textsc{Fermat} foliation~$\F_{0}^{d}.$
\begin{thmalph}\label{thmalph:Fermat}
{\sl Let $\F$ be a foliation of degree $d\geq2$ on $\pp.$ The following assertions are equivalent:
\begin{itemize}
\item [\textit{(1)}] $\F$ is linearly conjugated to the \textsc{Fermat} foliation $\F_{0}^{d}$;
\item [\textit{(2)}] $\F$ is reduced convex and $\mathcal{CS}(\F)=\{(1-d)^{\pm1}\}$;
\item [\textit{(3)}] $\F$ possesses three radial singularities of maximal order $d-1,$ necessarily non-aligned.
\end{itemize}
}
\end{thmalph}

\noindent In this theorem, the implication $\textit{(3)}\Rightarrow\textit{(1)}$ is a slight generalization of our previous result \cite[Proposition~6.3]{BM18}, where we had obtained the same conclusion but with the additional hypothesis that the three radial singularities of $\F$ are not aligned.

\begin{coralph}\label{coralph:CS=(1-d)pm1}
{\sl If $\mathcal{HCS}_{d}=\{(1-d)^{\pm1}\}$ then, up to automorphisms of $\pp$, the \textsc{Fermat} foliation $\F_{0}^{d}$ is the unique reduced convex foliation in degree $d.$}
\end{coralph}

\noindent The following theorem gives equivalent conditions for a foliation of degree $5$ on $\pp$ to be conjugated to the \textsc{Hilbert} modular foliation $\Hilbertcinq.$
\begin{thmalph}\label{thmalph:Hilbert5}
{\sl Let $\F$ be a foliation of degree $5$ on $\pp.$ The following assertions are equivalent:
\begin{itemize}
\item [\textit{(1)}] $\F$ is linearly conjugated to the \textsc{Hilbert} modular foliation $\Hilbertcinq$;
\item [\textit{(2)}] $\F$ is reduced convex and $\mathcal{CS}(\F)=\{-\frac{3}{2}\pm\frac{\sqrt{5}}{2}\}$;
\item [\textit{(3)}] $\F$ possesses three radial singularities $m_1,m_2,m_3$ of order~$3$ (necessarily non-aligned) and two radial singularities of order~$1$ on each invariant line $(m_jm_l),1\leq j<l\leq3.$
\end{itemize}
}
\end{thmalph}

\vspace{-1cm}
\begin{figure}[h]
\begin{center}
\begin{turn}{-18}
\includegraphics[width=7cm]{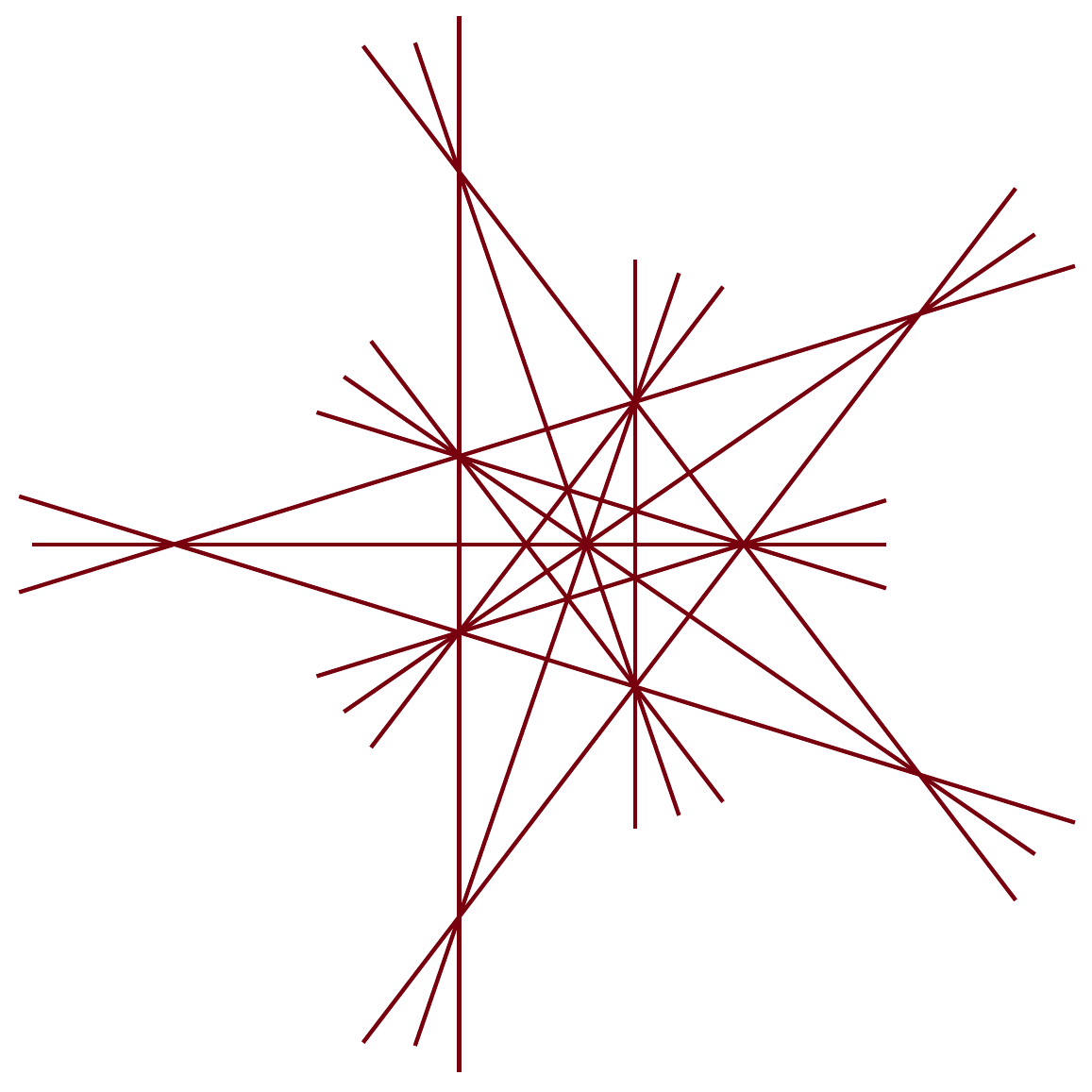}
\end{turn}
\end{center}
\vspace{-1cm}
\caption{Arrangement of invariant lines of the \textsc{Hilbert} modular foliation $\Hilbertcinq$ which possesses $6$ radial singularities of order $3$, $10$ radial singularities of order $1$ and $15$ non-radial singularities with \textsc{Baum-Bott} invariant $-1$. Through each radial singularity of order $k\ge 1$ pass $k+2$ invariant lines.}
\end{figure}

\noindent Using essentially Theorems~\ref{thmalph:class-homogene-convexe-5},~\ref{thmalph:Fermat},~\ref{thmalph:Hilbert5} and Proposition~\ref{propalph:CS(F)-CS(d)}, we establish the following theorem.
\begin{thmalph}\label{thmalph:convexe-reduit-5}
{\sl Up to automorphisms of $\pp$ the \textsc{Fermat} foliation $\F_{0}^{5}$ and the \textsc{Hilbert} modular foliation $\Hilbertcinq$ are the only reduced convex foliations of degree five on $\pp.$}
\end{thmalph}
\bigskip
\smallskip

\section{Proof of the main results}\label{sec:Demonstration des principaux resultats}
\bigskip

\noindent We need to know the numbers\, $r_{ij}$\, of radial singularities of order $j$ of the homogeneous foliations $\mathcal{H}_{\hspace{0.2mm}i},$~\, $i=1,\ldots,14$, $j=1,\ldots,4,$ and the values of the \textsc{Camacho-Sad} indices $\mathrm{CS}(\mathcal{H}_{\hspace{0.2mm}i},\elltext_{\infty},s)$, $s\in\Sing(\mathcal{H}_{\hspace{0.2mm}i})\cap \elltext_{\infty}$, $i=1,\ldots,14$. For this reason, we have computed, for each $i=1,\ldots,14$, the type $\mathcal{T}_{\mathcal{H}_{\hspace{0.2mm}i}}$ of $\mathcal{H}_{\hspace{0.2mm}i}$ and the following polynomial (called \textsl{\textsc{Camacho-Sad} polynomial of the homogeneous foliation}  $\mathcal{H}_{\hspace{0.2mm}i}$)
\begin{align*}
\mathrm{CS}_{\mathcal{H}_{\hspace{0.2mm}i}}(\lambda)=\prod\limits_{s\in\Sing(\mathcal{H}_{\hspace{0.2mm}i})\cap \ellindice_{\infty}}(\lambda-\mathrm{CS}(\mathcal{H}_{\hspace{0.2mm}i},\elltext_{\infty},s)).
\end{align*}
\noindent Table~\ref{tab:CS(lambda)} below summarizes the types and  \textsc {Camacho-Sad} polynomials of the foliations  $\mathcal{H}_{\hspace{0.2mm}i}$, $i=1,\ldots,14$.
\begingroup
\renewcommand*{\arraystretch}{1.8}
\begin{table}[h]
\begin{center}
\begin{tabular}{|c|c|c|}\hline
$i$  & $\mathcal{T}_{\mathcal{H}_{\hspace{0.2mm}i}}$              &  $\mathrm{CS}_{\mathcal{H}_{\hspace{0.2mm}i}}(\lambda)$ \\\hline
$1$  & $2\cdot\mathrm{R}_4$                                       &  $(\lambda-1)^2(\lambda+\frac{1}{4})^4$\\\hline
$2$  & $1\cdot\mathrm{R}_1+1\cdot\mathrm{R}_3+1\cdot\mathrm{R}_4$ &  $\frac{1}{491}(\lambda-1)^3(491\lambda^3+982\lambda^2+463\lambda+64)$\\\hline
$3$  & $2\cdot\mathrm{R}_2+1\cdot\mathrm{R}_4$                    &  $(\lambda-1)^3(\lambda+\frac{3}{7})^2(\lambda+\frac{8}{7})$\\\hline
$4$  & $1\cdot\mathrm{R}_2+2\cdot\mathrm{R}_3$                    &  $(\lambda-1)^3(\lambda+\frac{9}{11})^2(\lambda+\frac{4}{11})$\\\hline
$5$  & $2\cdot\mathrm{R}_1+1\cdot\mathrm{R}_2+1\cdot\mathrm{R}_4$ &  $(\lambda-1)^4(\lambda+\frac{3}{2})^2$\\\hline
$6$  & $2\cdot\mathrm{R}_1+1\cdot\mathrm{R}_2+1\cdot\mathrm{R}_4$ &  $\frac{1}{59}(\lambda-1)^4(59\lambda^2+177\lambda+64)$\\\hline
$7$  & $2\cdot\mathrm{R}_1+2\cdot\mathrm{R}_3$                    &  $(\lambda-1)^4(\lambda^2+3\lambda+1)$\\\hline
$8$  & $2\cdot\mathrm{R}_1+2\cdot\mathrm{R}_3$                    &  $(\lambda-1)^4(\lambda+\frac{3}{2})^2$\\\hline
$9$  & $1\cdot\mathrm{R}_1+2\cdot\mathrm{R}_2+1\cdot\mathrm{R}_3$ &  $\frac{1}{197}(\lambda-1)^4(197\lambda^2+591\lambda+302-10\sqrt{61})$\\\hline
$10$ & $1\cdot\mathrm{R}_1+2\cdot\mathrm{R}_2+1\cdot\mathrm{R}_3$ &  $\frac{1}{197}(\lambda-1)^4(197\lambda^2+591\lambda+302+10\sqrt{61})$\\\hline
$11$ & $4\cdot\mathrm{R}_2$                                       &  $(\lambda-1)^4(\lambda+\frac{3}{2})^2$\\\hline
$12$ & $2\cdot\mathrm{R}_1+3\cdot\mathrm{R}_2$                    &  $(\lambda-1)^5(\lambda+4)$\\\hline
$13$ & $4\cdot\mathrm{R}_1+1\cdot\mathrm{R}_4$                    &  $(\lambda-1)^5(\lambda+4)$\\\hline
$14$ & $3\cdot\mathrm{R}_1+1\cdot\mathrm{R}_2+1\cdot\mathrm{R}_3$ &  $(\lambda-1)^5(\lambda+4)$\\\hline
\end{tabular}
\end{center}
\bigskip
\caption{Types and \textsc{Camacho-Sad} polynomials of the homogeneous foliations~$\mathcal{H}_1,\ldots,\mathcal{H}_{14}.$}\label{tab:CS(lambda)}
\end{table}
\endgroup

\begin{proof}[\sl Proof of Theorem~\ref{thmalph:class-homogene-convexe-5}]
Let $\mathcal{H}$ be a homogeneous convex foliation of degree $5$ on $\pp$, defined in the affine chart $(x,y)$, by the $1$-form $$\hspace{1cm}\omega=A(x,y)\mathrm{d}x+B(x,y)\mathrm{d}y,\quad A,B\in\mathbb{C}[x,y]_5,\hspace{2mm}\gcd(A,B)=1.$$ By \cite[Remark~2.5]{Bed17} the foliation  $\mathcal{H}$ cannot have $5+1=6$ distinct radial singularities; in other words it cannot be of one of the two types $5\cdot\mathrm{R}_{1}+1\cdot\mathrm{R}_{3}$ or $4\cdot\mathrm{R}_{1}+2\cdot\mathrm{R}_{2}.$ We are then in one of the following situations:
\begin{align*}
&\mathcal{T}_{\mathcal{H}}=2\cdot\mathrm{R}_{4}\hspace{0.5mm};&&
\mathcal{T}_{\mathcal{H}}=1\cdot\mathrm{R}_{1}+1\cdot\mathrm{R}_{3}+1\cdot\mathrm{R}_{4}\hspace{0.5mm};&&
\mathcal{T}_{\mathcal{H}}=2\cdot\mathrm{R}_{2}+1\cdot\mathrm{R}_{4}\hspace{0.5mm};\\
&
\mathcal{T}_{\mathcal{H}}=1\cdot\mathrm{R}_{2}+2\cdot\mathrm{R}_{3}\hspace{0.5mm};&&
\mathcal{T}_{\mathcal{H}}=2\cdot\mathrm{R}_{1}+1\cdot\mathrm{R}_{2}+1\cdot\mathrm{R}_{4}\hspace{0.5mm};&&
\mathcal{T}_{\mathcal{H}}=2\cdot\mathrm{R}_{1}+2\cdot\mathrm{R}_{3}\hspace{0.5mm};\\
&
\mathcal{T}_{\mathcal{H}}=1\cdot\mathrm{R}_{1}+2\cdot\mathrm{R}_{2}+1\cdot\mathrm{R}_{3}\hspace{0.5mm};&&
\mathcal{T}_{\mathcal{H}}=4\cdot\mathrm{R}_{2}\hspace{0.5mm};&&
\mathcal{T}_{\mathcal{H}}=2\cdot\mathrm{R}_{1}+3\cdot\mathrm{R}_{2}\hspace{0.5mm};\\
&
\mathcal{T}_{\mathcal{H}}=4\cdot\mathrm{R}_{1}+1\cdot\mathrm{R}_{4}\hspace{0.5mm};&&
\mathcal{T}_{\mathcal{H}}=3\cdot\mathrm{R}_{1}+1\cdot\mathrm{R}_{2}+1\cdot\mathrm{R}_{3}.
\end{align*}
The proof consists in to analyze these $11$ possibilities either by applying some results in \cite{BM18} or by appealing to a specific classification taken from \cite{CGNPP15}.

\begin{itemize}
\item [$\bullet$] We know from \cite[Propositions~4.1,~4.2]{BM18} that if a homogeneous convex foliation of degree $d\geq3$ on $\pp$ is of type $2\cdot\mathrm{R}_{d-1},$\, resp. $1\cdot\mathrm{R}_{\nu}+1\cdot\mathrm{R}_{d-\nu-1}+1\cdot\mathrm{R}_{d-1}$ with $\nu\in\{1,2,\ldots,d-2\},$ then it is linearly conjugated to the foliation $\mathcal{H}_{1}^{d}$, resp. $\mathcal{H}_{3}^{d,\nu},$ given by
    \[
    \omega_1^{\hspace{0.2mm}d}=y^d\mathrm{d}x-x^d\mathrm{d}y,
    \qquad\qquad\text{resp}.\hspace{1.5mm}
    \omega_{3}^{\hspace{0.2mm}d,\nu}=\sum\limits_{i=\nu+1}^{d}\binom{{d}}{{i}}x^{d-i}y^i\mathrm{d}x-\sum\limits_{i=0}^{\nu}\binom{{d}}{{i}}x^{d-i}y^i\mathrm{d}y.
    \]
    It follows that if the foliation $\mathcal{H}$ is of type $\mathcal{T}_{\mathcal{H}}=2\cdot\mathrm{R}_{4}$,\, resp.
    $\mathcal{T}_{\mathcal{H}}=1\cdot\mathrm{R}_{1}+1\cdot\mathrm{R}_{3}+1\cdot\mathrm{R}_{4}$,\, resp.
    $\mathcal{T}_{\mathcal{H}}=2\cdot\mathrm{R}_{2}+1\cdot\mathrm{R}_{4}$, then the $1$-form $\omega$ is linearly conjugated to
    \[
    \hspace{-3.79cm}\omega_1^{\hspace{0.2mm}5}=y^5\mathrm{d}x-x^5\mathrm{d}y=\omega_1,
    \]
    \begin{eqnarray*}
    \text{resp}.\hspace{1.5mm}
    \omega_{3}^{\hspace{0.2mm}5,1}
    \hspace{-2.24mm}&=&\hspace{-2.24mm}
    \sum\limits_{i=2}^{5}\binom{{5}}{{i}}x^{5-i}y^i\mathrm{d}x-\sum\limits_{i=0}^{1}\binom{{5}}{{i}}x^{5-i}y^i\mathrm{d}y=\omega_2,
    \end{eqnarray*}
    \begin{eqnarray*}
    \text{resp}.\hspace{1.5mm}
    \omega_{3}^{\hspace{0.2mm}5,2}
    \hspace{-2.24mm}&=&\hspace{-2.24mm}
    \sum\limits_{i=3}^{5}\binom{{5}}{{i}}x^{5-i}y^i\mathrm{d}x-\sum\limits_{i=0}^{2}\binom{{5}}{{i}}x^{5-i}y^i\mathrm{d}y=\omega_3.
    \end{eqnarray*}
\item [$\bullet$] Assume that $\mathcal{T}_{\mathcal{H}}=1\cdot\mathrm{R}_{2}+2\cdot\mathrm{R}_{3}$. This means that the rational map $\Gunderline_{\mathcal{H}}\hspace{1mm}\colon\mathbb{P}^{1}_{\mathbb{C}}\rightarrow \mathbb{P}^{1}_{\mathbb{C}}$, $\Gunderline_{\mathcal{H}}(z)=-\dfrac{A(1,z)}{B(1,z)},$ possesses three fixed critical points, one of multiplicity $2$ and two of multiplicity $3.$ By \cite[page~79]{CGNPP15}, $\Gunderline_{\mathcal{H}}$ is conjugated by a \textsc{Möbius} transformation to $z\mapsto-\dfrac{z^4(3z-5)}{5z-3}$. As a result, $\omega$ is linearly conjugated to $\omega_4$.
    \vspace{2mm}
\item [$\bullet$] Let us study the possibility $\mathcal{T}_{\mathcal{H}}=2\cdot\mathrm{R}_{1}+1\cdot\mathrm{R}_{2}+1\cdot\mathrm{R}_{4}.$ Up to linear conjugation we can assume that, for some $\alpha\in\C\setminus\{0,1\},$ the points $[1:0:0],\,[0:1:0],\,[1:1:0],\,[1:\alpha:0]\in\pp$ are radial singularities of $\mathcal{H}$ with respective orders $4,2,1,1,$ or equivalently that the points $\infty=[1:0],\,[0:1],\,[1:1],\,[1:\alpha]\in\mathbb{P}^{1}_{\mathbb{C}}$ are fixed and critical for $\Gunderline_{\mathcal{H}}$ with respective multiplicities $4,2,1,1.$ By~\cite[Lemma~3.9]{BM18}, there exist constants $a_0,a_2,b\in\C^{*},a_1\in\C$ such that
    \begin{align*}
    &\hspace{3mm}B(x,y)=bx^5,&& A(x,y)=(a_0\hspace{0.2mm}x^2+a_1xy+a_2y^2)y^3,&&
    (z-1)^2\hspace{1mm} \text{divides}\hspace{1mm} P(z),&&
    (z-\alpha)^2\hspace{1mm} \text{divides}\hspace{1mm} Q(z),
    \end{align*}
    where $P(z):=A(1,z)+B(1,z)$ and $Q(z):=A(1,z)+\alpha B(1,z).$ A straightforward computation leads to
    \begin{align*}
    &a_0=\dfrac{5a_2\alpha}{3}, && a_1=-\dfrac{5a_2(\alpha+1)}{4}, && b=-\dfrac{a_2(5\alpha-3)}{12}, && (\alpha+1)(3\alpha^2-5\alpha+3)=0.
    \end{align*}
    Replacing $\omega$ by $\dfrac{\raisebox{-0.8mm}{$12$}}{\raisebox{1mm}{$a_2$}}\omega,$ we reduce it to
    \[
    \hspace{1cm}\omega=y^3(20\alpha\hspace{0.1mm}x^2-15(\alpha+1)xy+12y^2)\mathrm{d}x-(5\alpha-3)x^5\mathrm{d}y,\qquad (\alpha+1)(3\alpha^2-5\alpha+3)=0.
    \]
    This $1$-form is linearly conjugated to one of the two $1$-forms $\omega_5$ or $\omega_6$. Indeed, on the one hand, if $\alpha=-1$, then $\omega_5=-\frac{1}{4}\omega.$ On the other hand, if $3\alpha^2-5\alpha+3=0$, then
    \[
    \omega_6=\frac{121(15\alpha-16)}{81(3\alpha-8)^5}\varphi^*\omega,\quad\text{where}\hspace{1.5mm}\varphi=\Big((3\alpha-8)x,-3y\Big).
    \]
\item [$\bullet$] Assume that $\mathcal{T}_{\mathcal{H}}=2\cdot\mathrm{R}_{1}+2\cdot\mathrm{R}_{3}$. Then the rational map $\Gunderline_{\mathcal{H}}$ admits four fixed critical points, two of multiplicity $1$ and two of multiplicity $3.$ This implies, by \cite[page~79]{CGNPP15}, that up to conjugation by a \textsc{Möbius} transformation, $\Gunderline_{\mathcal{H}}$ writes as
    \[
    \hspace{1cm}z\mapsto-\frac{z^4(3z+4cz-5c-4)}{z+c},
    \]
    where $c=-1/2\pm\sqrt{5}/10$\, or \,$c=-3/10\pm\sqrt{21}/10.$ Thus, up to linear conjugation
    \[
    \hspace{1cm}\omega=y^4(3y+4cy-5cx-4x)\mathrm{d}x+x^4(y+cx)\mathrm{d}y,\qquad c\in\left\{-\frac{1}{2}\pm\frac{\sqrt{5}}{10},\,-\frac{3}{10}\pm\frac{\sqrt{21}}{10}\right\}.
    \]
    In the case where $c=-1/2\pm\sqrt{5}/10$,\, resp. $c=-3/10\pm\sqrt{21}/10$, the $1$-form $\omega$ is linearly conjugated to $\omega_7$, resp. $\omega_8$. Indeed, on the one hand, if $c=-1/2+\sqrt{5}/10$,\, resp. $c=-3/10+\sqrt{21}/10$, then $\omega_7=-2(5-2\sqrt{5})\omega$, resp. $\omega_8=-10(9-2\sqrt{21})\omega$. On the other hand, if $c=-1/2-\sqrt{5}/10$,\, resp. $c=-3/10-\sqrt{21}/10$, then
    \begin{align*}
    &\hspace{0.9cm}\omega_7=-(25+11\sqrt{5})\varphi^*\omega,\hspace{1.2mm}\quad \text{where}\hspace{1.5mm} \varphi=\left(\tfrac{3-\sqrt{5}}{2}x,y\right),\\
    &\text{resp}.\hspace{1.5mm}
    \omega_8=5(87+19\sqrt{21})\psi^*\omega,\quad \text{where}\hspace{1.5mm} \psi=\left(\tfrac{\sqrt{21}-5}{2}x,y\right).
    \end{align*}
\item [$\bullet$] We know from \cite[page~79]{CGNPP15} that up to \textsc{Möbius} transformation there are two rational maps of degree~$5$ from the \textsc{Riemann} sphere to itself having four distinct fixed critical points, one of multiplicity~$1$, two of multiplicity~$2$ and one of multiplicity~$3$; thus up to automorphisms of $\pp$ there are two homogeneous convex foliations of degree~$5$ on $\pp$ having type $1\cdot\mathrm{R}_{1}+2\cdot\mathrm{R}_{2}+1\cdot\mathrm{R}_{3}.$ Now, by Table~\ref{tab:CS(lambda)}, we have on the one hand $\mathrm{CS}_{\mathcal{H}_{\hspace{0.2mm}9}}\neq\mathrm{CS}_{\mathcal{H}_{\hspace{0.2mm}10}},$ so that the foliations $\mathcal{H}_{\hspace{0.2mm}9}$ and $\mathcal{H}_{\hspace{0.2mm}10}$ are not linearly conjugated, and on the other hand $\mathcal{T}_{\mathcal{H}_{9}}=\mathcal{T}_{\mathcal{H}_{10}}=1\cdot\mathrm{R}_{1}+2\cdot\mathrm{R}_{2}+1\cdot\mathrm{R}_{3}.$
    It follows that if the foliation $\mathcal{H}$ is of type $\mathcal{T}_{\mathcal{H}}=1\cdot\mathrm{R}_{1}+2\cdot\mathrm{R}_{2}+1\cdot\mathrm{R}_{3},$ then $\mathcal{H}$ is linearly conjugated to one of the two foliations $\mathcal{H}_{\hspace{0.2mm}9}$ or $\mathcal{H}_{\hspace{0.2mm}10}.$
    \vspace{2mm}
\item [$\bullet$] Assume that $\mathcal{T}_{\mathcal{H}}=4\cdot\mathrm{R}_{2}$. The rational map $\Gunderline_{\mathcal{H}}$ has therefore four different fixed critical points of multiplicity $2$. By \cite[page~80]{CGNPP15}, up to conjugation by a \textsc{Möbius} transformation, $\Gunderline_{\mathcal{H}}$ writes as
    \[
    \hspace{1cm}z\mapsto-\frac{z^3(z^2-5z+5)}{5z^2-10z+4}.
    \]
    As a consequence, up to linear conjugation
    \[
    \hspace{2cm}\omega=y^3(5x^2-5xy+y^2)\mathrm{d}x+x^3(4x^2-10xy+5y^2)\mathrm{d}y.
    \]
    This $1$-form is linearly conjugated to
    \[
    \omega_{11}=\frac{\raisebox{-0.8mm}{$1$}}{\raisebox{0.5mm}{$8$}}\varphi^*\omega,\quad\text{where}\hspace{1.5mm}\varphi=(x+y,2y).
    \]
\item [$\bullet$] Assume that $\mathcal{T}_{\mathcal{H}}=2\cdot\mathrm{R}_{1}+3\cdot\mathrm{R}_{2}$. Then the rational map $\Gunderline_{\mathcal{H}}$ possesses five fixed critical points, two of multiplicity~$1$ and three of multiplicity~$2.$ By \cite[page~80]{CGNPP15}, $\Gunderline_{\mathcal{H}}$ is conjugated by a \textsc{Möbius} transformation to $z\mapsto-\dfrac{z^3(z^2+5z-20)}{20z^2-5z-1}$, which implies that $\omega$ is linearly conjugated to $\omega_{12}$.
\item [$\bullet$] Let us consider the eventuality $\mathcal{T}_{\mathcal{H}}=4\cdot\mathrm{R}_{1}+1\cdot\mathrm{R}_{4}.$ Up to isomorphism, we can assume that, for some~$\alpha,\beta\in\mathbb{C}\setminus\{0,1\}$ with $\alpha\neq\beta,$ the points $\infty=[1:0],\,[0:1],\,[1:1],\,[1:\alpha],\,[1:\beta]\in\mathbb{P}^{1}_{\mathbb{C}}$ are fixed and critical for $\Gunderline_{\mathcal{H}},$ with respective multiplicities $4,1,1,1,1.$ By \cite[Lemma~3.9]{BM18}, there exist constants $a_0,a_3,b\in\C^{*},a_1,a_2\in\C$ such that
    \begin{align*}
    &B(x,y)=bx^5,&&
    A(x,y)=(a_0\hspace{0.2mm}x^3+a_1x^2y+a_2xy^2+a_3y^3)y^2,&&
    (z-1)^2\hspace{1mm}       \text{divides}\hspace{1mm} P(z),\\
    &(z-\alpha)^2\hspace{1mm} \text{divides}\hspace{1mm} Q(z),&&
    (z-\beta)^2\hspace{1mm}   \text{divides}\hspace{1mm} R(z),
    \end{align*}
    where $P(z):=A(1,z)+B(1,z)$, $Q(z):=A(1,z)+\alpha B(1,z)$ and $R(z):=A(1,z)+\beta B(1,z).$ A straightforward computation gives us
    \begin{small}
    \begin{align*}
    \hspace{6mm}&b=\dfrac{a_3\alpha^2(\alpha-1)^2}{2(\alpha^2-\alpha+1)},&&
    a_0=-\dfrac{a_3\alpha(\alpha+1)(3\alpha^2-5\alpha+3)}{2(\alpha^2-\alpha+1)},&&
    a_1=\dfrac{a_3(\alpha^4+2\alpha^3-3\alpha^2+2\alpha+1)}{\alpha^2-\alpha+1},\\
    \hspace{6mm}&\beta=\dfrac{(\alpha+1)(3\alpha^2-5\alpha+3)}{5(\alpha^2-\alpha+1)},&&
    a_2 =-\dfrac{a_3(\alpha+1)(4\alpha^2-5\alpha+4)}{2(\alpha^2-\alpha+1)},&&
    (\alpha^2-2\alpha+2)(2\alpha^2-2\alpha+1)(\alpha^2+1)=0.
    \end{align*}
    \end{small}
    \hspace{-1mm}Multiplying $\omega$ by $\dfrac{\raisebox{-0.8mm}{$2$}}{\raisebox{1mm}{$a_3$}}(\alpha^2-\alpha+1),$ we reduce it to
    \begin{align*}
    &\omega=-y^2\Big(\alpha(\alpha+1)(3\alpha^2-5\alpha+3)x^3+(\alpha+1)(4\alpha^2-5\alpha+4)xy^2-2(\alpha^2-\alpha+1)y^3\Big)\mathrm{d}x\\
    &\hspace{6.8mm}+2(\alpha^4+2\alpha^3-3\alpha^2+2\alpha+1)x^2y^3\mathrm{d}x+\alpha^2(\alpha-1)^2x^5\mathrm{d}y,
    \end{align*}
    with $(\alpha^2-2\alpha+2)(2\alpha^2-2\alpha+1)(\alpha^2+1)=0.$ This $1$-form $\omega$ is linearly conjugated to
    \[
    \omega_{13}=-\frac{(\alpha+1)(3\alpha^2-5\alpha+3)}{5\alpha^3(\alpha-1)^4}\varphi^*\omega,
    \quad\text{where}\hspace{1.5mm}
    \varphi=\left(x,\frac{5\alpha\left(\alpha-1\right)^2}{\left(\alpha+1\right)\left(3\alpha^2-5\alpha+3\right)}y\right).
    \]
\item [$\bullet$] Finally let us examine the case $\mathcal{T}_{\mathcal{H}}=3\cdot\mathrm{R}_{1}+1\cdot\mathrm{R}_{2}+1\cdot\mathrm{R}_{3}.$ Up to linear conjugation we can assume that the points $\infty=[1:0],\,[0:1],\,[1:1],\,[1:\alpha],\,[1:\beta]\in\mathbb{P}^{1}_{\mathbb{C}},$ where $\alpha\beta\in\mathbb{C}\setminus\{0,1\}$ and $\alpha\neq\beta,$ are fixed and critical for $\Gunderline_{\mathcal{H}},$ with respective multiplicities $3,2,1,1,1.$ A similar reasoning as in the previous case leads to
    \begin{align*}
    &\hspace{7mm}\omega=\omega(\alpha)=y^3\Big(\big(\alpha^2-3\alpha+1\big)\big(\alpha^2+5\alpha+1\big)x^2-2\big(\alpha+1\big)\big(\alpha^2-5\alpha+1\big)xy
    +\big(\alpha^2-7\alpha+1\big)y^2\Big)\mathrm{d}x\\
    &\hspace{2.65cm}+\alpha\hspace{0.3mm}x^4\Big(2\alpha\big(\alpha^2-\alpha+1\big)x-\big(\alpha+1\big)\big(3\alpha^2-5\alpha+3\big)y\Big)\mathrm{d}y,
    \end{align*}
    with $P(\alpha)=0$ where $P(z):=3z^6-39z^5+194z^4-203z^3+194z^2-39z+3.$ The $1$-form $\omega$ is linearly conjugated to
    \begin{align*}
    &\hspace{7mm}\omega_{14}=y^3\Big(\big(\sigma^2-3\sigma+1\big)\big(\sigma^2+5\sigma+1\big)x^2-2\big(\sigma+1\big)\big(\sigma^2-5\sigma+1\big)xy
    +\big(\sigma^2-7\sigma+1\big)y^2\Big)\mathrm{d}x\\
    &\hspace{1.68cm}+\sigma\hspace{0.3mm}x^4\Big(2\sigma\big(\sigma^2-\sigma+1\big)x-\big(\sigma+1\big)\big(3\sigma^2-5\sigma+3\big)y\Big)\mathrm{d}y,
    \end{align*}
    where $\sigma=\rho+\mathrm{i}\sqrt{\frac{1}{6}-\frac{4}{3}\rho-\frac{1}{3}\rho^2}$\, and $\rho$ is the unique real number satisfying $8\rho^3-52\rho^2+134\rho-15=0.$ Indeed, on the one hand, it is easy to see that $\sigma$ is a root of the polynomial $P$, so that $\omega_{14}=\omega(\sigma).$ On~the~other hand, a straightforward computation shows that if $\alpha_1$ and $\alpha_2$ are any two roots of $P$ then
    \begin{small}
    \begin{align*}
    \omega(\alpha_2)=-\frac{\mu}{21600}\left(13035\alpha_1^5-167802\alpha_1^4+821633\alpha_1^3-777667\alpha_1^2+743778\alpha_1-76185\right)
    \varphi^*\big(\omega(\alpha_1)\big)
    \end{align*}
    \end{small}
    \hspace{-1mm}with \begin{small}$\mu=195\alpha_2^4-202\alpha_2^3+233\alpha_2^2-42\alpha_2+3$\end{small},\, \begin{small}$\varphi=\left(x,-\dfrac{\lambda}{43200}y\right)$\end{small} where
    \begin{Small}
    \begin{align*}
    &\hspace{4mm}
    \lambda=\left(39\alpha_2^5-501\alpha_2^4+2447\alpha_2^3-2293\alpha_2^2+2343\alpha_2-477\right)
    \left(24\alpha_1^5-309\alpha_1^4+1510\alpha_1^3-1415\alpha_1^2+1446\alpha_1-21\right).
    \end{align*}
    \end{Small}
\end{itemize}
\smallskip

\noindent The foliations $\mathcal{H}_1,\ldots,\mathcal{H}_{14}$ are not linearly conjugated because we have $\mathcal{T}_{\mathcal{H}_{\hspace{0.2mm}i}}~\neq~\mathcal{T}_{\mathcal{H}_{j}}$ or $\mathrm{CS}_{\mathcal{H}_{\hspace{0.2mm}i}}\neq\mathrm{CS}_{\mathcal{H}_{j}}$ for all $i,j\in\{1,\ldots,14\}$ with $i\neq j$  (see Table~\ref{tab:CS(lambda)}). This ends the proof Theorem~\ref{thmalph:class-homogene-convexe-5}.
\end{proof}

\noindent Let $\F$ be a reduced convex foliation of degree $d\geq1$ on $\pp$ and let $\elltext$ be one of its  $3d$ invariant lines. To the pair $(\F,\elltext)$ we can associate thanks to \cite{BM19Z} a homogeneous convex foliation $\mathcal{H}_{\F}^{\ellindice}$ of degree $d$ on $\pp,$ called {\sl homogeneous degeneration of $\F$ along $\elltext,$} as follows. Let us fix homogeneous coordinates $[x:y:z]\in\pp$ such that $\elltext=(z=0)$; since  $\elltext$ is $\F$-invariant, $\F$ is described in the affine chart $z=1$ by a $1$-form $\omega$ of type $$\omega=\sum_{i=0}^{d}(A_i(x,y)\mathrm{d}x+B_i(x,y)\mathrm{d}y),$$ where $A_i,\,B_i$ are homogeneous polynomials of degree $i$.  By~\cite[Proposition~3.2]{BM19Z} we have $\gcd(A_d,B_d)=1$ which allows us to define the foliation $\mathcal{H}_{\F}^{\ellindice}$ by the $1$-form $$\omega_d=A_d(x,y)\mathrm{d}x+B_d(x,y)\mathrm{d}y.$$
It is easy to check that this definition is intrinsic, {\it i.e.} it does not depend on the choice of the homogeneous coordinates $[x:y:z]$ nor on the choice of the $1$-form $\omega$ describing $\F.$

\noindent The following result, taken from \cite[Proposition~3.2]{BM19Z}, will be very useful to us.
\begin{pro}[\cite{BM19Z}]\label{prop:F-dégénère-H}
{\sl With the previous notations, the foliation  $\mathcal{H}_{\F}^{\ellindice}$ has the following properties:
\vspace{1mm}
\begin{itemize}
\item [\textit{(i)}] $\mathcal{H}_{\F}^{\ellindice}$ belongs to the \textsc{Zariski} closure of the $\mathrm{Aut}(\pp)$-orbit of~$\F;$
\item [\textit{(ii)}] $\elltext$ is  invariant by $\mathcal{H}_{\F}^{\ellindice}$;
\vspace{0.5mm}
\item [\textit{(iii)}] $\Sing(\mathcal{H}_{\F}^{\ellindice})\cap\elltext=\Sing(\F)\cap\elltext$;
\vspace{0.5mm}
\item [\textit{(iv)}] every singular point of $\mathcal{H}_{\F}^{\ellindice}$ on $\elltext$ is non-degenerate;
\vspace{0.5mm}
\item [\textit{(v)}] a point $s\in\elltext$ is a radial singularity of order $k\leq d-1$ for $\mathcal{H}_{\F}^{\ellindice}$ if and only if it is for $\F$;
\vspace{0.5mm}
\item [\textit{(vi)}] $\forall\hspace{1mm}s\in\Sing(\mathcal{H}_{\F}^{\ellindice})\cap\elltext,
              \hspace{1mm}
              \mathrm{CS}(\mathcal{H}_{\F}^{\ellindice},\elltext,s)=\mathrm{CS}(\F,\elltext,s).$
\end{itemize}
}
\end{pro}

\begin{proof}[\sl Proof of Proposition~\ref{propalph:CS(F)-CS(d)}]
Since by hypothesis $\F$ is reduced convex, all its singularities are non-degenerate (\cite[Lemma~6.8]{BM18}). Let $\elltext$ be an invariant line of $\F.$ By~\cite[Proposition~2.3]{Bru15} it follows thseeat $\F$ possesses exactly $d+1$ singularities on $\elltext.$ The \textsc{Camacho}-\textsc{Sad} formula (see~\cite{CS82}) $\sum\limits_{s\in\Sing(\F)\cap \ellindice}\mathrm{CS}(\F,\elltext,s)=1$ then implies the existence of $s\in\Sing(\F)\cap \elltext$ such that $\mathrm{CS}(\F,\elltext,s)\in\C\setminus\{0,1\}$; as a result $\mathcal{CS}(\F)\neq\emptyset.$

\noindent Let $\lambda\in\mathcal{CS}(\F)\subset\C\setminus\{0,1\}$; there is a line $\elltext_1$ invariant by $\F$ and a singular point $s\in\elltext_1$ of $\F$ such that $\mathrm{CS}(\F,\elltext_1,s)=\lambda.$ By \cite[Lemma~3.1]{BM19Z} through the point $s$ passes a second $\F$-invariant line $\elltext_{2}.$ Since $\mathrm{CS}(\F,\elltext_1,s)\mathrm{CS}(\F,\elltext_2,s)=1,$ we have $\mathrm{CS}(\F,\elltext_2,s)=\frac{1}{\lambda}$; thus $\frac{1}{\lambda}\in\mathcal{CS}(\F).$ Moreover, by \cite[Proposition~3.2]{BM19Z} (\emph{cf.}~assertion \textit{(vi)}~of~Proposition~\ref{prop:F-dégénère-H} above), we have
\begin{align*}
&
\mathrm{CS}(\mathcal{H}_{\F}^{\ellindice_1},\elltext_1,s)=\mathrm{CS}(\F,\elltext_1,s)=\lambda
&&\text{and}&&
\mathrm{CS}(\mathcal{H}_{\F}^{\ellindice_2},\elltext_2,s)=\mathrm{CS}(\F,\elltext_2,s)=\frac{1}{\lambda},
\end{align*}
which shows that $\lambda\in\mathcal{HCS}_d,$ hence $\mathcal{CS}(\F)\subset\mathcal{HCS}_d.$
\end{proof}

\noindent An immediate consequence of Table~\ref{tab:CS(lambda)} is the following:
\begin{cor}\label{cor:CS5}
{\sl $\mathcal{HCS}_5=\{-4^{\pm 1},-\frac{3}{2}\pm\frac{\sqrt{5}}{2}\}=\mathcal{CS}(\F_0^5)\cup\mathcal{CS}(\Hilbertcinq).$
}
\end{cor}

\noindent The proof of Theorem~\ref{thmalph:Fermat} uses Lemma~\ref{lem:3-singularites-radiales-non-alignees} and Lemma~\ref{lem:CS(H)=(1-d)pm1} stated below.
\begin{lem}\label{lem:3-singularites-radiales-non-alignees}
{\sl Let  $\F$ be a foliation of degree $d\geq2$ on $\pp$ having two radial singularities $m_1,m_2$ of maximal order $d-1.$ Then the line $(m_1m_2)$ cannot contain a third radial singularity of $\F.$
}
\end{lem}

\begin{proof}
Let us choose homogeneous coordinates $[x:y:z]\in\pp$ such that $m_1=[0:1:0]$ and $m_2=[1:0:0].$ Thanks to \cite[Proposition~2.2]{Bru15} (\emph{cf.} \cite[Remark~1.2]{Bed17}) the line  $\ell=(m_1m_2)$ must be invariant by  $\F.$ Then~the foliation  $\F$ is given in the affine chart $z=1$ by a $1$-form $\omega$ of type $\omega=\omega_0+\omega_1+\cdots+\omega_d,$ where, for $0\leq i\leq d,\omega_i=A_i(x,y)\mathrm{d}x+B_i(x,y)\mathrm{d}y,$ with $A_i,B_i$ homogeneous polynomials of degree $i.$

\noindent Writing explicitly that the points $m_j, j=1,2,$ are radial singularities of maximal order $d-1$ of $\F$ (see \cite[Proposition~6.3]{BM18}), we obtain that the highest degree homogeneous part $\omega_d$ of $\omega$ is of the form $\omega_d=ay^d\mathrm{d}x+bx^d\mathrm{d}y,$ with $a,b\in\C^*.$ Thus, $\omega_d$ defines a homogeneous convex foliation $\mathcal{H}$ of degree $d$ on $\pp$ of type $\mathcal{T}_{\mathcal{H}}=2\cdot\mathrm{R}_{d-1}.$ If we would know that $\F$ was a convex reduced foliation then $\mathcal H=\mathcal H^\ell_\F$ for the invariant line  $\ell=(m_1m_2)$  and we could apply Proposition~\ref{prop:F-dégénère-H} to conclude. Anyway, reasoning as in the proof of~\cite[Proposition 6.4]{BM18}, we see that $\F $ and $\mathcal{H}$ have the same~singularities on the line $(m_1m_2)$ and that every singularity $s$ of $\F$ on $(m_1m_2)$ distinct from $ m_1$ and $m_2$ is non-degenerate and has \textsc{Camacho}-\textsc{Sad} index $\mathrm{CS}(\F,(m_1m_2),s)=\mathrm{CS}(\mathcal{H},(m_1m_2),s)=\frac{1}{1-d}\neq1,$ hence the lemma.
\end{proof}

\begin{lem}\label{lem:CS(H)=(1-d)pm1}
{\sl Let $\mathcal H$ be a homogeneous convex foliation of degree $d\geq2$ on $\pp.$ Assume that every non radial singularity $s$ of $\mathcal{H}$ on $\elltext_\infty$ has \textsc{Camacho}-\textsc{Sad} index $\mathrm{CS}(\mathcal H,\elltext_\infty,s)\in\{(1-d)^{\pm 1}\}.$ Denote by $\kappa_0$ the number of (distinct) radial singularities of $\mathcal{H}$ and by $\kappa_1$ (resp. $\kappa_2$) the number of singularities $s\in\elltext_\infty$ of $\mathcal{H}$ such that  $\mathrm{CS}(\mathcal H,\elltext_\infty,s)=1-d$ (resp. $\mathrm{CS}(\mathcal H,\elltext_\infty,s)=\frac{1}{1-d}$). Then
\begin{itemize}
  \item [--] either $(\kappa_0,\kappa_1,\kappa_2)=(d,1,0)$;
  \item [--] or $(\kappa_0,\kappa_1,\kappa_2)=(2,0,d-1),$ in which case $\mathcal{T}_{\mathcal{H}}=2\cdot\mathrm{R}_{d-1}.$
\end{itemize}
}
\end{lem}

\noindent Before proving this lemma let us make two remarks:
\begin{rem}\label{rem:homogene-convexe-d+1-singularites}
By \cite[Theorem~4.3]{Cra06} every homogeneous convex foliation of degree~$d$ on the complex projective plane has exactly~$d+1$ singularities on the line at infinity, necessarily non-degenerate.
\end{rem}

\begin{rem}\label{rem:radiale-equivaut-CS=1}
A straightforward computation shows that if a homogeneous foliation $\mathcal{H}$ on $\pp$ possesses a non-degenerate singularity $s\in\elltext_\infty$ such that $\mathrm{CS}(\mathcal{H},\elltext_\infty,s)=1,$ then $s$ is necessarily radial. In particular, when $\mathcal{H}$ is convex, a singularity $s\in\elltext_\infty$ of $\mathcal{H}$ is radial if and only if it has \textsc{Camacho}-\textsc{Sad} index $\mathrm{CS}(\mathcal{H},\elltext_\infty,s)=1.$
\end{rem}

\begin{proof}[\sl Proof of Lemma~\ref{lem:CS(H)=(1-d)pm1}]
The \textsc{Camacho}-\textsc{Sad} formula $\sum\limits_{s\in\Sing(\mathcal H)\cap \ellindice_\infty}\mathrm{CS}(\mathcal H,\elltext_\infty,s)=1$  (see~\cite{CS82}) and Remarks~\ref{rem:homogene-convexe-d+1-singularites} and~\ref{rem:radiale-equivaut-CS=1} imply that
\begin{align*}
&\kappa_0+\kappa_1+\kappa_2=d+1 &&\text{and} &&\kappa_0+(1-d)\kappa_1+\frac{\kappa_2}{1-d}=1.
\end{align*}
From these two equations we obtain that $\kappa_0=2+\kappa_1(d-2)$\, and \,$\kappa_2=(d-1)(1-\kappa_1)\geq0,$ so that $\kappa_1\in\{0,1\},$ hence the lemma.
\end{proof}

\begin{proof}[\sl Proof of Theorem~\ref{thmalph:Fermat}] The implication $\textit{(3)}\Rightarrow\textit{(1)}$ follows from \cite[Proposition~6.3]{BM18} and from Lemma~\ref{lem:3-singularites-radiales-non-alignees}.

\noindent The fact that \textit{(1)} implies \textit{(2)} follows from the reduced convexity of the foliation $\F_0^d$ and from the equality  $\mathcal{CS}(\F_0^d)=\{(1-d)^{\pm 1}\}$ (Remark~\ref{rem:CS(Hilbertcinq)-CS(Fermat)}).

\noindent Let us show that \textit{(2)} implies \textit{(3)}. Assume that $\F$ is reduced convex and that $\mathcal{CS}(\F)=\{(1-d)^{\pm 1}\}.$ Let~$m$ be a non radial singular point of $\F$; through $m$ pass exactly two $\F$-invariant lines $\elltext_{m}^{(1)}$ and $\elltext_{m}^{(2)}$ (\cite[Lemma~3.1]{BM19Z}). It~follows that $\mathrm{CS}(\F,\elltext_{m}^{(i)},m)=(1-d)^{\pm 1}$ for $i=1,2.$ Up to renumbering the $\elltext_{m}^{(i)},$ we can assume that $\mathrm{CS}(\F,\elltext_{m}^{(1)},m)=\frac{1}{1-d}$ and $\mathrm{CS}(\F,\elltext_{m}^{(2)},m)=1-d$ for any choice of the non radial singularity $m\in\Sing\F.$ Moreover, according to Proposition~\ref{prop:F-dégénère-H}, for any invariant line $\elltext$ of $\F$ and for any non radial singularity $s\in\elltext$ of the homogeneous degeneration $\mathcal{H}_{\F}^{\ellindice}$ of $\F$ along $\elltext,$ we have $\mathrm{CS}(\mathcal{H}_{\F}^{\ellindice},\elltext,s)=\mathrm{CS}(\F,\elltext,s)\in\C\setminus\{0,1\}$ and therefore $\mathrm{CS}(\mathcal{H}_{\F}^{\ellindice},\elltext,s)\in\{(1-d)^{\pm 1}\}.$ It follows by Lemma~\ref{lem:CS(H)=(1-d)pm1} that $\mathcal{H}_{\F}^{\ellindice_{m}^{(1)}}$ is of type $2\cdot\mathrm{R}_{d-1}$. This implies, according to assertion~\textit{(v)} of~Proposition~\ref{prop:F-dégénère-H}, that~$\F$ possesses two radial singularities $m_1,m_2$ of maximal order $d-1$ on the line~$\elltext_{m}^{(1)}.$ Let $m'$ be another non radial singular point of $\F$ not belonging to the line $\elltext_{m}^{(1)}.$ For any $s\in\Sing\F$ let us denote, as in~\cite[Section~1]{BM18}, by $\tau(\F,s)$ the tangency order of $\F$ with a generic line passing through $s.$ For~$i=1,2$ we have $\tau(\F,m')+\tau(\F,m_i)=1+d>\deg\F,$ which implies (\emph{cf.}~\cite[Proposition~2.2]{Bru15}) that the lines $(m'm_i)$ are invariant by $\F.$ Thus, the line $\elltext_{m'}^{(1)}$ is one of the lines $(m'm_1)$ or $(m'm_2)$ and it in turn contains another radial singularity $m_3$ of maximal order $d-1$ of $\F.$
\end{proof}

\noindent The proof of Theorem~\ref{thmalph:Hilbert5} uses the following lemma for $d=5$ which we state in arbitrary degree $d$ as it could be used in other situations. It can be proved in the same way as in \cite[Proposition~6.3]{BM18}.
\begin{lem}\label{lem:3-singularites-radiales-d-2}
{\sl Let $\F$ be a foliation of degree $d\ge 3$ on $\pp.$ Assume that the points $m_1=[0:0:1],\,m_2=[1:0:0]$\, and \,$m_3=[0:1:0]$ are radial singularities of order $d-2$ of $\F.$ Let $\omega$ be a $1$-form defining $\F$ in the affine chart $z=1.$ Then $\omega$ is of the form
\begin{align*}
&\omega=(x\mathrm{d}y-y\mathrm{d}x)(\lambda_{0,0}+\lambda_{1,0}x+\lambda_{0,1}y+\lambda_{1,1}xy)+y^{d-2}(a_{1,0}x+a_{0,1}y+a_{1,1}xy+a_{0,2}y^2)\mathrm{d}x\\
&\hspace{0.7cm}+x^{d-2}(b_{1,0}x+b_{0,1}y+b_{1,1}xy+b_{2,0}x^2)\mathrm{d}y,
\end{align*}
where $\lambda_{i,j},a_{i,j},b_{i,j}\in\mathbb{C}$ with $\lambda_{0,0}\neq0.$
}
\end{lem}

\begin{proof}[\sl Proof of Theorem~\ref{thmalph:Hilbert5}]
The implication $\textit{(1)}\Rightarrow\textit{(2)}$ follows from the reduced convexity of the foliation $\Hilbertcinq$ and from the equality  $\mathcal{CS}(\Hilbertcinq)=\{-\frac{3}{2}\pm\frac{\sqrt{5}}{2}\}$ (Remark~\ref{rem:CS(Hilbertcinq)-CS(Fermat)}).

\noindent Let us show that \textit{(2)} implies \textit{(3)}. Assume that $\F$ is reduced convex and that $\mathcal{CS}(\F)=\{-\frac{3}{2}\pm\frac{\sqrt{5}}{2}\}.$ Let~$\elltext$~be an invariant line of  $\F.$ The homogeneous foliation  $\mathcal{H}_{\F}^{\ellindice}$ -- homogeneous degeneration of $\F$ along $\elltext$ -- being convex of degree $5$, it must be linearly conjugated to one of the fourteen homogeneous foliations given by Theorem~\ref{thmalph:class-homogene-convexe-5}. Moreover, let $m$ be a non radial singular point of $\F$ on $\elltext$; then we have $\mathrm{CS}(\F,\elltext,m)=-\frac{3}{2}\pm\frac{\sqrt{5}}{2}.$ According to Proposition~\ref{prop:F-dégénère-H}, the point $m$ is also a non radial singularity for $\mathcal{H}_{\F}^{\ellindice}$ and we have $\mathrm{CS}(\mathcal{H}_{\F}^{\ellindice},\elltext,m)=\mathrm{CS}(\F,\elltext,m)=-\frac{3}{2}\pm\frac{\sqrt{5}}{2}.$ It then follows from Table~\ref{tab:CS(lambda)} that $\mathcal{H}_{\F}^{\ellindice}$ is of type $2\cdot\mathrm{R}_1+2\cdot\mathrm{R}_3.$ This implies, according to assertion~\textit{(v)} of~Proposition~\ref{prop:F-dégénère-H}, that $\F$ has exactly four radial singularities on the line $\elltext;$ two of them $m_1,m_2$ are of order~$3$ and the other two are of order~$1.$ Let us consider another $\F$-invariant line $\elltext'\neq\elltext$ passing through $m_1,$ whose existence is guaranteed by \cite[Lemma~3.1]{BM19Z}. Then $\elltext'$ contains another radial singularity $m_3$ of order $3$ of $\F$ and two radial singularities of order $1$ of $\F.$ By \cite[Proposition~2.2]{Bru15}, the fact that $\tau(\F,m_2)+\tau(\F,m_3)=4+4>\deg\F$ ensures the $\F$-invariance of the line $\elltext''=(m_2m_3)$. Therefore $\elltext''$ in turn contains two radial singularities of order $1$ of $\F.$

\noindent Finally, let us prove that \textit{(3)} implies \textit{(1)}. Assume that \textit{(3)} holds. Then there is a homogeneous coordinate system $[x:y:z]\in\pp$ in which $m_1=[0:0:1],\,m_2=[1:0:0]$\, and \,$m_3=[0:1:0].$
Moreover, in this coordinate system the lines $x=0$, $y=0$, $z=0$ must be invariant by $\F$ and there exist $x_0,y_0,z_0,x_1,y_1,z_1\in \mathbb{C}^{*},$ $x_1\neq x_0,y_1\neq y_0,z_1\neq z_0,$ such that the points $m_4=[x_0:0:1],$\, $m_5=[1:y_0:0],$\, $m_6=[0:1:z_0],$\, $m_7=[x_1:0:1],$\, $m_8=[1:y_1:0],$\, $m_9=[0:1:z_1]$ are radial singularities of order~$1$ of $\F.$ Let us set $\xi=\frac{x_1}{x_0},\,\rho=\frac{y_1}{y_0},\,\sigma=\frac{z_1}{z_0},\,w_0=x_0y_0z_0$; then $w_0\in\mathbb{C}^*$, $\xi,\rho,\sigma\in\mathbb{C}\setminus\{0,1\}$ and, up to renumbering the $x_i,y_i,z_i,$ we can assume that $\xi,$ $\rho$ and $\sigma$ are all of modulus greater than or equal to~$1.$ Let $\omega$ be a $1$-form defining $\F$ in the affine chart $z=1.$ By~conjugating $\omega$ by the diagonal linear transformation $(x_0\hspace{0.2mm}x,\hspace{0.2mm}x_0y_0y)$, we reduce ourselves to $m_4=[1:0:1],$\, $m_5=[1:1:0],$\, $m_6=[0:1:w_0],$\, $m_7=[\xi:0:1],$\, $m_8=[1:\rho:0],$\, $m_9=[0:1:\sigma w_0].$ Since $m_1$, $m_2$ and $m_3$ are radial singularities of order $3$, $\omega$ can be written as in the expression given in Lemma~\ref{lem:3-singularites-radiales-d-2} in the case $d=5$. Then, as in the proof of \cite[Theorem~B]{BM19Z}, by writing explicitly that the points $m_j,4\leq j\leq9,$ are radial singularities of order $1$ of $\F$ we obtain that $w_0=\pm(\sqrt{5}-2)$ and

\begin{Small}
\begin{align*}
&
\xi=\rho=\sigma=\frac{3}{2}+\frac{\sqrt{5}}{2},&&
b_{1,0}=\frac{25+11\sqrt{5}}{2}a_{0,2},&&
\lambda_{0,0}=\frac{47+21\sqrt{5}}{2}a_{0,2},\\
&
a_{1,0}=(9+4\sqrt{5})(5w_0+5-2\sqrt{5})a_{0,2},&&
b_{0,1}=-\frac{(65+29\sqrt{5})(w_0+5-2\sqrt{5})}{2}a_{0,2},&&
\lambda_{1,0}=-\frac{65+29\sqrt{5}}{2}a_{0,2},\\
&
a_{0,1}=-\frac{25+11\sqrt{5}}{2}a_{0,2}w_0,&&
b_{1,1}=(5+2\sqrt{5})a_{0,2},&&
\lambda_{0,1}=-(85+38\sqrt{5})a_{0,2}w_0,\\
&
a_{1,1}=-\frac{5+\sqrt{5}}{2}a_{0,2},&&
b_{2,0}=-\frac{7+3\sqrt{5}}{2}a_{0,2},&&
\lambda_{1,1}=\frac{(47+21\sqrt{5})(5w_0+5-2\sqrt{5})}{2}a_{0,2}
\end{align*}
\end{Small}
\hspace{-1mm}with $a_{0,2}\neq0.$ Thus $\omega$ is of the form
\begin{align*}
&\omega=
\frac{a_{0,2}(47+21\sqrt{5})}{4}\Big(x\mathrm{d}y-y\mathrm{d}x\Big)\Big(2-\big(5-\sqrt{5}\big)x-w_0\big(5+\sqrt{5}\big)y+\big(10w_0+10-4\sqrt{5}\big)xy\Big)\\
&\hspace{0.72cm}+\frac{a_{0,2}}{2}y^3\Big(\big(9+4\sqrt{5}\big)\big(10w_0+10-4\sqrt{5}\big)x-w_0\big(25+11\sqrt{5}\big)y-\big(5+\sqrt{5}\big)xy+2y^2\Big)\mathrm{d}x\\
&\hspace{0.72cm}+\frac{a_{0,2}}{2}x^3\Big(\big(25+11\sqrt{5}\big)x-\big(65+29\sqrt{5}\big)\big(w_0+5-2\sqrt{5}\big)y-\big(7+3\sqrt{5}\big)x^2
+\big(10+4\sqrt{5}\big)xy\Big)\mathrm{d}y.
\end{align*}
The $1$-form $\omega$ is linearly conjugated to
\[
\omegahilbertcinq=\big(y^2-1\big)\big(y^2-(\sqrt{5}-2)^2\big)\big(y+\sqrt{5}x\big)\mathrm{d}x
                  -\big(x^2-1\big)\big(x^2-(\sqrt{5}-2)^2\big)\big(x+\sqrt{5}y\big)\mathrm{d}y.
\]
Indeed, if $w_0=\sqrt{5}-2$, resp. $w_0=2-\sqrt{5}$, then
\begin{align*}
&
\hspace{0.9cm}\omegahilbertcinq=\frac{32(3571-1597\sqrt{5})}{a_{0,2}}\varphi_{1}^*\omega,
\hspace{2.9mm}\quad\text{where}\hspace{1.5mm}
\varphi_1=\left(\frac{3+\sqrt{5}}{4}\big(x+1\big),-\frac{2+\sqrt{5}}{2}\big(y-1\big)\right),\\
&
\text{resp}.\hspace{1.5mm}
\omegahilbertcinq=\frac{32(64079-28657\sqrt{5})}{a_{0,2}}\varphi_{2}^*\omega,
\hspace{-1mm}\quad\text{where}\hspace{1.5mm}
\varphi_2=\left(\frac{2+\sqrt{5}}{2}\big(x+\sqrt{5}-2\big),-\frac{7+3\sqrt{5}}{4}\big(y+\sqrt{5}-2\big)\right).
\end{align*}
\end{proof}

\begin{proof}[\sl Proof of Theorem~\ref{thmalph:convexe-reduit-5}]
Let $\F$ be a reduced convex foliation of degree $5$ on $\pp.$ By assertion~\textit{(a)} of Proposition~\ref{propalph:CS(F)-CS(d)} and Corollary~\ref{cor:CS5} we have $\emptyset\neq\mathcal{CS}(\F)\subset\mathcal{HCS}_5=\left\{-4^{\pm 1},-\tfrac{3}{2}\pm\tfrac{\sqrt{5}}{2}\right\}.$ Hence, according to assertion~\textit{(b)} of Proposition~\ref{propalph:CS(F)-CS(d)}, one of the following three possibilities does occur:
\begin{itemize}
\item[(i)] $\mathcal{CS}(\F)=\{-4^{\pm 1}\};$

\item[(ii)] $\mathcal{CS}(\F)=\{-\frac{3}{2}\pm\frac{\sqrt{5}}{2}\};$

\item[(iii)] $\mathcal{CS}(\F)=\{-4^{\pm 1},-\frac{3}{2}\pm\frac{\sqrt{5}}{2}\}.$
\end{itemize}
In the case~(i) (resp.~(ii)) the foliation $\F$ is linearly conjugated to~$\F_0^5$ (resp.~$\Hilbertcinq$), thanks to Theorem~\ref{thmalph:Fermat} (resp.~Theorem~\ref{thmalph:Hilbert5}). To establish the theorem, it therefore suffices to exclude the possibility~(iii). Let us assume by contradiction that (iii) happens. Then $\F$ possesses two invariant lines $\elltext,\elltext'$ and two non radial singularities $m\in\elltext,m'\in\elltext'$ such that $\mathrm{CS}(\F,\elltext,m)=-\frac{1}{4}$ and $\mathrm{CS}(\F,\elltext',m')=-\frac{3}{2}\pm\frac{\sqrt{5}}{2}.$ According to Proposition~\ref{prop:F-dégénère-H}, the point $m$ (resp. $m'$) is also a non radial singularity for the homogeneous foliation $\mathcal{H}_{\F}^{\ellindice}$ (resp.~$\mathcal{H}_{\F}^{\ellindice'}$) and we have
\begin{align*}
&\mathrm{CS}(\mathcal{H}_{\F}^{\ellindice},\elltext,m)=\mathrm{CS}(\F,\elltext,m)=-\tfrac{1}{4}
&&\text{and}&&
\mathrm{CS}(\mathcal{H}_{\F}^{\ellindice'},\elltext',m')=\mathrm{CS}(\F,\elltext',m')=-\tfrac{3}{2}\pm\tfrac{\sqrt{5}}{2}.
\end{align*}
Moreover, as in the proof of Theorem~\ref{thmalph:Hilbert5}, each of the foliations $\mathcal{H}_{\F}^{\ellindice}$ and $\mathcal{H}_{\F}^{\ellindice'}$ is linearly conjugated to one of the fourteen homogeneous foliations given by Theorem~\ref{thmalph:class-homogene-convexe-5}. It then follows from Table~\ref{tab:CS(lambda)} that $\mathcal{H}_{\F}^{\ellindice}$ and $\mathcal{H}_{\F}^{\ellindice'}$ are respectively of types $2\cdot\mathrm{R}_4$ and $2\cdot\mathrm{R}_1+2\cdot\mathrm{R}_3.$ This implies, according to assertion~\textit{(v)} of~Proposition~\ref{prop:F-dégénère-H}, that $\F$ admits two radial singularities of order $4$ on the line $\elltext$ and four radial singularities on the line $\elltext',$ two of order~$1$ and two of order~$3.$ Let~$m_1$ (resp. $m_2$) be a radial singularity of order $4$ (resp. $3$) of $\F$ on the line $\elltext$ (resp. $\elltext'$). Since $\tau(\F,m_1)+\tau(\F,m_2)=5+4>\deg\F,$ the line $\elltext''=(m_1m_2)$ is invariant by $\F$ (\emph{cf.} \cite[Proposition~2.2]{Bru15}). The homogeneous foliation $\mathcal{H}_{\F}^{\ellindice''}$ being convex of degree $5$, it must therefore be of type $1\cdot\mathrm{R}_1+1\cdot\mathrm{R}_3+1\cdot\mathrm{R}_4$ so that it possesses a non radial singularity $m''$ on the line $\elltext''$ satisfying (see~Table~\ref{tab:CS(lambda)})
\begin{align*}
&\mathrm{CS}(\mathcal{H}_{\F}^{\ellindice''},\elltext'',m'')=\mathrm{CS}(\F,\elltext'',m'')=\lambda,
&&\text{with}&&
491\lambda^3+982\lambda^2+463\lambda+64=0
\end{align*}
which is impossible.
\end{proof}

\section{Conjectures}\label{sec:Conjectures}
\bigskip

\noindent The notion of convex reduced foliation has an interesting relation with certain line arrangements in $\pp.$ Indeed, according to \cite{PP16} we say that an arrangement $\mathcalsansmathptm{A}$ of $3d$ lines in $\pp$ has {\sl \textsc{Hirzebruch}'s property} if each line of $\mathcalsansmathptm{A}$ intersects the other lines of $\mathcalsansmathptm{A}$ in exactly $d+1$ points. The $3d$ invariant lines of a reduced convex foliation of degree $d$ on $\pp$ form a line arrangement which satisfies \textsc{Hirzebruch}'s property, thanks to \cite[Lemma~6.8]{BM18}~and~\cite[Lemma~3.1]{BM19Z}.
\noindent The expected conjectural picture for the reduced convex foliations on $\mathbb P^2_\C$ is the following: besides the \textsc{Fermat} foliations $\F_0^d$, with $\mathcal{CS}(\F_0^d) = \{(1 -d)^{\pm 1}\}$, there exist special reduced convex foliations only for $d = 4,5$ and $d = 7$, namely, the \textsc{Hesse} pencil in degree $4$, and the two \textsc{Hilbert} modular foliations in degree $5$ and $7$ presented in the Introduction, for which
\[\mathcal{CS}(\F_H^d)=\left\{\begin{array}{ll}\{-1\}& \text{for $d=4$,}\\[1mm] \{-\frac{3}{2}\pm\frac{\sqrt{5}}{2}\}& \text{for $d=5$,}\\[1mm] \{-\big(\frac{3}{4}\big)^{\pm1}\}& \text{for $d=7$.}
\end{array}\right.\]

\noindent {\it i.e.} we expect that there are no other convex reduced foliations on $\mathbb P^2_\C$ and for this reason we propose:

\begin{conj}\label{conj:CS=(1-d)pm1}
{\sl We have
\[\mathcal{HCS}_d=\left\{\begin{array}{ll}
\{(1-d)^{\pm 1}\} & \text{for $2\le d\neq 4,5,7,$}\\[1mm]
\{(1-d)^{\pm 1}\}\cup\mathcal{CS}(\F_H^d) & \text{for $d=4,5,7.$}
\end{array}\right.\]
}
\end{conj}

\noindent This conjecture, combined with Corollary~\ref{coralph:CS=(1-d)pm1}, would imply a negative answer in degree $d\neq7$ to \cite[Problem~9.1]{MP13} as we have already shown for $d\le 5$.
\smallskip

\noindent To every rational map $f\hspace{1mm}\colon\mathbb{P}^{1}_{\mathbb{C}}\rightarrow \mathbb{P}^{1}_{\mathbb{C}}$ and to every integer $d\geq 2,$ we associate respectively the following subsets of $\C\setminus\{0,1\}$:
\begin{itemize}
\item $\mathcal{M}(f)$ is, by definition, the set of $\mu\in\C\setminus\{0,1\}$ such that there is a fixed point $p$ of $f$ satisfying $f'(p)=\mu;$

\item $\mathcal{M}_d$ is defined as the set of $\mu\in\C\setminus\{0,1\}$ for which there exist critically fixed rational maps $f_1,f_2\hspace{1mm}\colon\mathbb{P}^{1}_{\mathbb{C}}\rightarrow\mathbb{P}^{1}_{\mathbb{C}}$ of degree $d$ having respective fixed points $p_1$ and $p_2$ such that $f_1'(p_1)=\mu$ and $f_2'(p_2)=\frac{\mu}{\mu-1}.$
\end{itemize}
\smallskip

\noindent The introduction of the sets $\mathcal{M}(f)$ and $\mathcal{M}_d$ is motivated by the following remark.
\begin{rem}
Let $\mathcal{H}$ be a homogeneous foliation of degree $d$ on $\pp.$ According to \cite[Section~2]{BM18} the point $s=[b:a:0]\in \elltext_{\infty}$ is a non-degenerate singularity of $\mathcal{H}$ if and only if the point $p=[a:b]\in\mathbb{P}^{1}_{\mathbb{C}}$  is fixed by $\Gunderline_{\mathcal{H}}$ with multiplier $\Gunderline_{\mathcal{H}}'(p)\neq1,$ in which case the \textsc{Camacho}-\textsc{Sad} index $\mathrm{CS}(\mathcal{H},\elltext_\infty,s)$ coincides with the index $\imath(\Gunderline_{\mathcal{H}},p)$ of $\Gunderline_{\mathcal{H}}$ at the fixed point $p$:
\[
\mathrm{CS}(\mathcal{H},\elltext_\infty,s)=\imath(\Gunderline_{\mathcal{H}},p):=\frac{1}{2\mathrm{i}\pi}\int_{|z-p|=
\varepsilon}\frac{\mathrm{d}z}{z-\Gunderline_{\mathcal{H}}(z)}=
\frac{1}{1-\Gunderline_{\mathcal{H}}'(p)}.
\]
Thus, the map $\mu\mapsto\frac{1}{1-\mu}$ sends $\mathcal M(\Gunderline_{\mathcal{H}})$ (resp. $\mathcal M_d$) bijectively onto $\mathcal{CS}(\mathcal{H})$ (resp. $\mathcal{HCS}_d$).
\end{rem}

\noindent Using the above definition of the sets $\mathcal{M}(f),$ Theorem~4.3 of E.~\textsc{Crane} \cite{Cra06} can be reformulated as follows:
\begin{thm}[\textsc{Crane}, \cite{Cra06}]
{\sl Let $f\hspace{1mm}\colon\mathbb{P}^{1}_{\mathbb{C}}\rightarrow \mathbb{P}^{1}_{\mathbb{C}}$ be a critically fixed rational map of degree  $d\geq2.$ Let $n\leq d$ denote the number of (distinct) critical points of $f.$ Then

\noindent\textbf{\textit{1.}} $f$ has exactly $d+1$ fixed points, of which $d+1-n$ are non-critical;

\noindent\textbf{\textit{2.}} the set $\mathcal{M}(f)$ is contained in $\C\setminus\left(\overline{\mathbb{D}}\left(0,1\right)\cup\mathbb{D}\left(1+\rho,\rho\right)\right),$ where $\overline{\mathbb{D}}(0,1)$ denotes the closed unit disk of $\C$ and $\mathbb{D}(1+\rho,\rho)\subset\C$ the open disk of radius $\rho=\frac{1}{d+n-2}$ and center $1+\rho.$ Moreover, $\mu\in\mathcal{M}(f)$ belongs to the boundary of the disk $\mathbb{D}(1+\rho,\rho)$ if and only if $n=d,$ in which case $\mu=\frac{d}{d-1}.$
}
\end{thm}

\noindent This theorem translates in terms of homogeneous foliations as follows:
\begin{cor}\label{cor:Crane}
{\sl
Let $\mathcal{H}$ be a homogeneous convex foliation of degree $d\geq2$ on $\pp.$ Let $n=\deg\mathcal{T}_{\mathcal{H}}$ denote the number of (distinct) radial singularities of $\mathcal{H}.$ Then

\noindent\textbf{\textit{1.}} $\mathcal{H}$ has exactly $d+1$ singularities on the line at infinity, of which $d+1-n$ are non radial;

\noindent\textbf{\textit{2.}} for any non radial singularity $s\in\elltext_\infty$ of $\mathcal{H},$ we have $-\frac{1}{2}<-\mathrm{Re}\big(\mathrm{CS}(\mathcal H,\elltext_\infty,s)\big)\leq \frac{d+n}{2}-1.$ This last inequality is an equality if and only if $n=d$, in which case $\mathrm{CS}(\mathcal H,\elltext_\infty,s)=1-d.$
}
\end{cor}

\noindent With the notations of Corollary~\ref{cor:Crane}, since $n\leq d$ we have in particular $-\frac{1}{2}<-\mathrm{Re}\big(\mathrm{CS}(\mathcal H,\elltext_\infty,s)\big)\leq d-1$. According to Remark~\ref{rem:CS(Hilbertcinq)-CS(Fermat)}, the value $d-1$ is attained by $(\mathcal{H},s)\mapsto-\mathrm{Re}\big(\mathrm{CS}(\mathcal H,\elltext_\infty,s)\big)$. However, after having checked many examples, we think that the lower bound $-\frac{1}{2}$ of $-\mathrm{Re}\big(\mathrm{CS}(\mathcal H,\elltext_\infty,s)\big)$ is not optimal and we propose the following conjecture with the value $\frac{1}{d-1}$ which is also attained by $(\mathcal{H},s)\mapsto-\mathrm{Re}\big(\mathrm{CS}(\mathcal H,\elltext_\infty,s)\big)$  (Remark~\ref{rem:CS(Hilbertcinq)-CS(Fermat)}).
\begin{conj}\label{conj:-Re(CS)}
{\sl
If $\mathcal{H}$ is a homogeneous convex foliation of degree $d\geq2$ on $\pp,$ then for any non radial singularity $s\in\elltext_\infty$ of $\mathcal{H}$ we have $\frac{1}{d-1}\leq-\mathrm{Re}\big(\mathrm{CS}(\mathcal H,\elltext_\infty,s)\big).$ Alternatively, if $f\hspace{1mm}\colon\mathbb{P}^{1}_{\mathbb{C}}\rightarrow \mathbb{P}^{1}_{\mathbb{C}}$ is a critically fixed rational map of degree $d\geq2,$ then the set $\mathcal{M}(f)$ is contained in the closed disk $\overline{\mathbb{D}}\left(\frac{d+1}{2},\frac{d-1}{2}\right)\subset\C$ of center $\frac{d+1}{2}$ and radius $\frac{d-1}{2}.$
}
\end{conj}

\begin{figure}[h]
\begin{center}
\includegraphics[width=7cm]{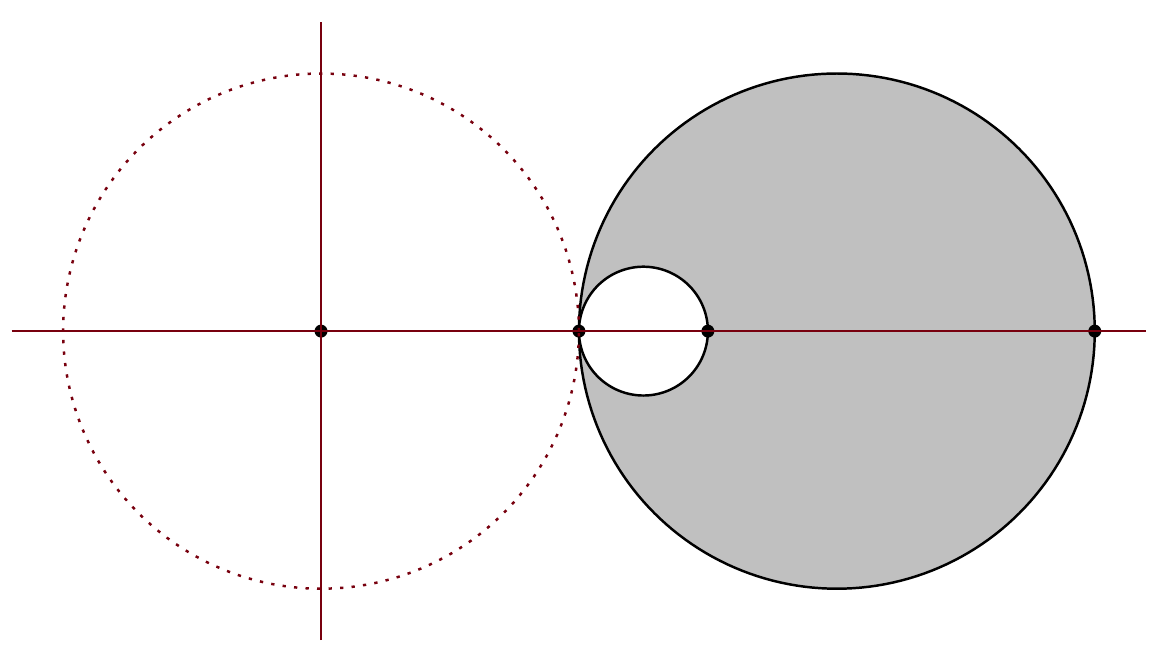}
\end{center}
\caption{The set $\mathcal{M}(f)$ is conjectured to be contained in the grey region for any critically fixed rational map $f\hspace{1mm}\colon\mathbb{P}^{1}_{\mathbb{C}}\rightarrow \mathbb{P}^{1}_{\mathbb{C}}$ of degree $d.$ It is known that it is contained in the exterior of the union of the dashed circle and the inner white disk. The black points from left to right are $0,1,\frac{d}{d-1}$ and $d$. Conjecture 1 for $2\le d\neq 4,5,7$ is equivalent to the statement $\mathcal M_d=\{\frac{d}{d-1},d\}$.}
\end{figure}

\noindent This conjecture is also motivated by the following remark:
\begin{rem}
If Conjecture~\ref{conj:-Re(CS)} is true, Conjecture~\ref{conj:CS=(1-d)pm1} claims that in degree $2\leq d\neq4,5,7$ the set $\mathcal{HCS}_{d}$ consists of the extreme values of $-\mathrm{Re}\big(\mathrm{CS}(\mathcal H,\elltext_\infty,s)\big)$ when $\mathcal H$ runs through the set of homogeneous convex foliations of degree $d$ on $\pp$ and $s$ runs through the set of non radial singularities of~$\mathcal H$ on the line $\elltext_\infty.$
\end{rem}

\noindent Elementary computations, using the normal forms of homogeneous convex foliations of degree $d\in\{2,3,4,5\}$ on $\pp$ presented in \cite[Proposition 7.4]{FP15}, \cite[Corollary~C]{Bed17}, \cite[Theorem~A]{BM19Z} and in Theorem~\ref{thmalph:class-homogene-convexe-5}, show the validity of Conjecture~\ref{conj:CS=(1-d)pm1} for $d\in\{2,3\}$ and Conjecture~\ref{conj:-Re(CS)} for $d\in\{2,3,4,5\}$. Moreover, very long computations carried out with \texttt{Maple} by the first author give $49$ normal forms for homogeneous convex foliations of degree $6$ on~$\pp$ and allow to verify the validity of Conjectures~\ref{conj:CS=(1-d)pm1} and \ref{conj:-Re(CS)} for $d=6.$ The more difficult case $d=7$ is out of reach at this moment.


\end{document}